\documentclass[a4paper,12pt]{article}

\usepackage{amssymb,latexsym,amsthm,amsmath,stmaryrd,amsfonts}
\usepackage{mathabx}
\usepackage{graphicx}
\usepackage{color}%
\usepackage{times}
\newcommand{\R}{\mathbb{R}}

\DeclareMathOperator{\dom}{dom}
\DeclareMathOperator{\ran}{ran}
\DeclareMathOperator{\gra}{gra}

\newtheorem{theorem}{Theorem}[section]
\newtheorem{lemma}[theorem]{Lemma}
\newtheorem{proposition}[theorem]{Proposition}
\newtheorem{corollary}[theorem]{Corollary}

\theoremstyle{definition}
\newtheorem{definition}[theorem]{Definition}
\newtheorem{example}[theorem]{Example}

\theoremstyle{remark}
\newtheorem{remark}[theorem]{Remark}

\newcommand{\tos}{\rightrightarrows} 

\newcommand{\inner}[2]{\langle #1,#2 \rangle}

\usepackage{hyperref}
\title{The Pseudomonotone Polar for Multivalued Operators}
\author{Orestes Bueno
	\thanks{Universidad del Pac\'ifico. Av. Salaverry 2020, Jes\'us Mar\'ia, Lima, Per\'u. Email: 
	\texttt{\{o.buenotangoa, cotrina\_je\}@up.edu.pe}} 
	\and John Cotrina\footnotemark[1]}

\DeclareMathOperator{\cone}{cone}
\DeclareMathOperator{\co}{co}

\newcounter{enumtemp}

\begin{document}
\maketitle

\begin{abstract}
In this work, we study the pseudomonotonicity of multivalued operators from the point of view of polarity, in an analogous way as the well-known monotone polar due to Mart\'inez-Legaz and Svaiter, and the quasimonotone polar recently 
introduced by Bueno and Cotrina.  
We show that this new polar, adapted for pseudomonotonicity, possesses analogous properties to the monotone and quasimonotone polar, among which are a characterization of pseudomonotonicity, maximality and pre-maximality. Furthermore, we characterize the notion of $D$-maximal pseudomonotonicity introduced by Hadjisavvas. 
We conclude this work studying the connections between pseudomonotonicity and variational inequality problems.

\bigskip


\bigskip

\end{abstract}

\section{Introduction}
The notion of pseudomonotonicity, in the sense of Karamardian~\cite{Karamardian1976}, plays a big role in certain mathematical applications, for instance, variational inequality problems and problems on consumer theory, see \cite{Crouzeix1997,Crouzeix2011,HadjSchai06,Hadjisavvas2012,John2001}, and the references therein. 

The theory of pseudomonotone operators has been steadily developed in the past decades. 
However, in contrast with the theory of maximal monotone operators, the literature about maximal generalized monotonicity is, to our best knowledge, reduced to \cite{Aussel2013-4,MR2577322,Eberhard07,Had03}.
In particular, in \cite{Had03}, the notion of $D$-maximal pseudomonotonicity is introduced, which is based on an equivalence related to variational inequalities. 


The monotone polar was introduced by Mart\'inez-Legaz and Svaiter~\cite{BSML}.
Its main feature is to provide an easy characterization of monotonicity, maximal monotonicity and pre-maximal monotonicity.  In a similar way, Bueno and Cotrina~\cite{BueCot16} introduced the quasimonotone polar which, besides  characterizing  (maximal, pre-maximal) quasimonotonicity, is related to the Minty variational inequality problem and, in particular, to the study of the adjusted normal operator defined by Aussel and Hadjisavvas~\cite{AuHad2005}.

The aim of this paper is to introduce the pseudomonotone polar and study its connection with the maximal pseudomonotonicty (including the weaker $D$-maximality) and its connection with the Minty and Stampacchia variational inequality problems.  The paper is organized as follows: in section 3 we provide the definition of pseudomonotone polar, along with a characterization in terms of normal cones and the study of its set of zeros.  In sections 4 and 5 we use the polar to study the pseudomonotonicity, maximal and pre-maximal pseudomonotonicity of an operator and, in addition, we characterize the $D$-maximality in terms of the pseudomonotone polar.  Finally in section 6, we relate the pseudomonotonicity with the solution sets of variational inequality problems.

\section{Preliminary definitions and notations}
Let $U,V$ be non-empty sets. 
A \emph{multivalued operator} $T:U\tos V$ is an application $T:U\to \mathcal{P}(V)$, that is, for $u\in U$, $T(u)\subset V$. 
The \emph{domain}, \emph{range} and \emph{graph} of $T$ are defined, respectively, as
\begin{gather*}
\dom(T)=\big\{u\in U\::\: T(u)\neq\emptyset\big\},\qquad \ran(T)=\bigcup_{u\in U}T(u),\\ 
\text{and }\gra(T)=\big\{(u,v)\in U\times V\::\: v\in T(u)\big\}.
\end{gather*}
Given $A\subset U$, the \emph{restriction} of $T$ to $A$ is the operator $T|_A$ defined as $T|_A(x)=T(x)$, 
if $x\in A$, and $T|_A(x)=\emptyset$, otherwise.

From now on, we will identify multivalued operators with their graphs, so we will write $(u,v)\in T$ instead of $(u,v)\in \gra(T)$.

Let $V$ be a real vector space, and let $A\subset V$. The \emph{convex hull} of $A$, denoted as $\co(A)$, is the smallest convex set (in the sense of inclusion) which contains $A$. The \emph{segment} between $u$ and $v$ in $V$ is the set $[u,v]=\co\{u,v\}$.
The \emph{conic hull} of $A$ is the set
\[
\cone(A)=\{tv\in V\::\:t\geq 0,\, v\in A\},
\]
whereas its \emph{strict conic hull} is the set
\[
\cone_{\circ}(A)=\{tv\in V\::\:t> 0,\, v\in A\}.
\]
Furthermore, given $T:U\tos V$, we denote $\cone(T):U\tos V$ (respectively, $\cone_{\circ}(T):U\tos V$) as the operator defined by $\cone(T)(x)=\cone(T(x))$ (respectively, $\cone_{\circ}(T)(x)=\cone_{\circ}(T(x))$), for all $x\in\dom(T)$. In addition, the \emph{set of zeros} of $T$ is the set
\[
Z_T=\{x\in X\::\: 0\in T(x)\}.
\]
Note that $Z_T=Z_{\cone_{\circ}(T)}$.
\newcommand{\cones}{\cone_{\circ}}

Let $X$ be a Banach space and $X^*$ be its topological dual. 
The \emph{duality product} is defined as $\inner{x}{x^*}=x^*(x)$.

Let $C\subset X$ be a convex set and $x\in X$. The \emph{normal cone} of $C$ at $x$ is the set
\[
N_{C}(x)=\{x^*\in X^*\::\: \inner{y-x}{x^*}\leq 0,\,\forall\, y\in C\}.
\]
Moreover, the \emph{strict normal cone} of $C$ at $x$ is the set
\[
N^{\circ}_{C}(x)=\{x^*\in X^*\::\: \inner{y-x}{x^*}< 0,\,\forall\, y\in C\}.
\]
By a vacuity argument, we note that if $C=\emptyset$ then $N_C(x)=N^{\circ}_C(x)=X^*$. 

The \emph{monotone polar}~\cite{BSML} of an operator $T:X\tos X^*$ is the operator $T^{\mu}$ defined as
\[
T^{\mu}=\{(x,x^*)\in X\times X^*\::\: \inner{x-y}{x^*-y^*}\geq 0,\,\forall\,(y,y^*)\in T\}.
\]
In the same way, the \emph{quasimonotone polar}~\cite{BueCot16} of $T$ is the operator $T^{\nu}$ defined as
\[
T^{\nu}=\{(x,x^*)\in X\times X^*\::\:\min\{\inner{y-x}{x^*},\inner{x-y}{y^*}\}\leq 0,\,\forall\,(y,y^*)\in T\}.
\]
The notions of \emph{(quasi)monotonicity}, \emph{maximal} and \emph{pre-maximal (quasi)monotonicity} can be expressed in terms of the (quasi)monotone polar. More precisely: 
\begin{enumerate}
\item $T$ is (quasi)monotone if, and only if, $T\subset T^{\mu} (T^\nu)$;
\item $T$ is maximal (quasi)monotone if, and only if, $T=T^{\mu}(T^\nu)$;
\item $T$ is pre-maximal (quasi)monotone if, and only if, $T$ and $T^{\mu} (T^\nu)$ are (quasi)monotone.
\end{enumerate}
%


\section{The pseudomonotone polar}
We say that $(x,x^*)$ is in a \emph{pseudomonotone relation} with $(y,y^*)$, denoted by $(x,x^*)\sim_p(y,y^*)$, if
\[
\min\{\inner{x-y}{y^*},\inner{y-x}{x^*}\}<0\quad\text{or}\quad\inner{x-y}{y^*}=\inner{y-x}{x^*}=0.
\]
It is clear that $\sim_p$ is a reflexive and symmetric relation on $X\times X^*$. Although it is not transitive, a weaker form of transitivity can be obtained.
\begin{proposition}\label{pro:pseudo-new}
	Let $(x,x^*),(y,y^*),(z,z^*)\in X\times X^*$, with $z\in[x,y]$. If $(x,x^*)\sim_p(z,z^*)$ and $(z,z^*)\sim_p(y,y^*)$ then $(x,x^*)\sim_p(y,y^*)$.
\end{proposition}
\begin{proof}
	Let $z=tx+(1-t)y$, with $t\in]0,1[$ (the result is trivial when $t=0,1$). Note that
	\begin{gather*}
	\inner{z-y}{y^*}=t\inner{x-y}{y^*},\qquad\inner{z-x}{x^*}=(1-t)\inner{y-x}{x^*}\\
		(1-t)\inner{y-z}{z^*}=-t\inner{x-z}{z^*}.
	\end{gather*}
	Without loss of generality, assume that $\inner{x-y}{y^*}>0$ and $\inner{y-x}{x^*}\geq 0$. Then $\inner{z-x}{x^*}\geq 0$ and $\inner{z-y}{y^*}>0$. As $(y,y^*)\sim_p(z,z^*)$, the latter implies that $\inner{y-z}{z^*}<0$, which in turn implies $\inner{x-z}{z^*}>0$. This contradicts $(x,x^*)\sim_p(z,z^*)$.
\end{proof}

The proof of the following lemma is analogous to Lemma 3.4 in~\cite{BueCot16}
\begin{lemma}\label{3.6}
Let $T:X\tos X^*$ be a multivalued operator and let $(x,x^*)\in T^{\rho}$, $(y,y^*)\in T$ and $t,s> 0$. Then $(x,tx^*)\sim_p(y,sy^*)$.
\end{lemma}

\begin{definition}
The \emph{pseudomonotone polar} of $T:X\tos X^*$, $T^{\rho}$ is given by
\[
T^{\rho}=\{(x,x^*)\in X\times X^*:~(x,x^*)\sim_p(y,y^*),~\forall (y,y^*)\in T\} 
\]
\end{definition}
It is clear from the definition that 
\begin{equation}\label{eq:polarsubset}
T^\mu\subset T^\rho\subset T^\nu, 
\end{equation}
for every operator $T:X\tos X^*$.
\begin{example}\label{exa:xtimes0}
For $x\in X$, $\{(x,0)\}^\rho=\{(x,0)\}^\mu$. Moreover, it is straightforward to verify that 
$(X\times \{0\})^\rho= X\times \{0\}$.
\end{example}

The following proposition lists some polarity properties of $T^{\rho}$.

\begin{proposition}\label{pro:polarp}\quad
\begin{enumerate}
\item $\displaystyle\left(\bigcup_{i\in I}T_i\right)^{\rho}=\bigcap_{i\in I}T_i^{\rho}$.
\item $T\subset T^{\rho\rho}$.
\item $T^{\rho\rho\rho}=T^{\rho}$.
\item If $T\subset S$ then $S^{\rho}\subset T^{\rho}$.
\setcounter{enumtemp}{\theenumi}
\end{enumerate}
In addition
\begin{enumerate}
\setcounter{enumi}{\theenumtemp}
\item $\emptyset^{\rho}=X\times X^*$, $(X\times X^*)^{\rho}=\emptyset$
\item $T^{\rho\mu}\subset T^{\mu\mu}\subset T^{\mu\rho}$
\item $T^{\rho\mu}\subset T^{\rho\rho}\subset T^{\mu\rho}$
\end{enumerate}
\end{proposition}
\begin{proof}
Items {\it 1} to {\it 4} are direct consequences of the fact that the map $T\mapsto T^{\rho}$ is a polarity~\cite{BSML}. Item {\it 5} follows directly from the definition and a vacuity argument. Items {\it 6} and {\it 7} come from equation~\eqref{eq:polarsubset} and item {\it 4}.   	
\end{proof}



We now deal with the set of zeros of $T^{\rho}$.

\begin{proposition}\label{pro:zeros}
Let $T:X\tos X^*$ be a multivalued operator. Then $Z_{T^\rho}$ is a weak-closed convex set.
\end{proposition}
\begin{proof}
It is enough to observe that $(x,0)\in T^{\rho}$ if, and only if, $\inner{x-y}{y^*}\leq 0$,  for all $(y,y^*)\in T$. Therefore, definining $C(y,y^*)=\{x\in X:~\langle x-y,y^*\rangle\leq0\}$, we have
\[
 Z_{T^\rho}= \bigcap_{(y,y^*)\in T}C(y,y^*).
\]
The proposition follows, since each $C(y,y^*)$ is weak-closed and convex.
\end{proof}
\begin{proposition}
Let $T:X\tos X^*$ be a multivalued operator. Then $Z_{T^\mu}=Z_{T^\rho}$. 
\end{proposition}
\begin{proof}
Since $T^\mu\subset T^\rho$, $Z_{T^\mu}\subset Z_{T^\rho}$. On the other hand, given $x\in Z_{T^{\rho}}$, from the proof of Proposition~\ref{pro:zeros}, we have $\inner{x-y}{y^*}\leq 0$, for all $(y,y^*)\in T$. This in turn implies
\[
\inner{x-y}{0-y^*}\geq 0,\qquad\forall\,(y,y^*)\in T,
\]
that is, $(x,0)\in T^{\mu}$. This proves the proposition.
\end{proof}

Corollary~3.8 in \cite{BueCot16} shows that the images of the quasimonotone polar are conic, convex and weak$^*$-closed.  In a similar way, we have the following proposition.
\begin{proposition}
Let $T:X\tos X^*$ be multivalued operator. Then the following assertions hold.
\begin{enumerate}
\item $\cones(T)^{\rho}=T^{\rho}=\cones(T^{\rho})$. In particular $T^{\rho}(x)$ is a cone, for all $x\in X$;
\item $T^{\rho}(x)$ is a convex set, for all $x\in X$;
\end{enumerate}
\end{proposition}

 A natural question is if the closure of the pseudomonotone polar coincides with the quasimonotone polar of an operator. 
The answer is negative, take for instance $T=\R\times\{0\}$. It is clear that $T={\rm cl(T^\rho)}\neq T^\nu=\R^2$.

Given $T:X\tos X^*$ a multivalued operator, consider the sets
\[
V_T(x)=\{y\in X\::\: \exists\,y^*\in T(y),\,\inner{x-y}{y^*}>0\}.
\]
and
\[
W_T(x)=\{y\in X\::\: \exists\,y^*\in T(y),\,\inner{x-y}{y^*}=0\}.
\]
Thus, we can state the following characterization of $T^{\rho}(x)$.
\begin{proposition}\label{pro:trhogen}
Let $T:X\tos X^*$ be a multivalued operator. Then, for $x\in X$,
\[
T^{\rho}(x)=N^{\circ}_{V_T(x)}(x)\cap N_{W_T(x)}(x).
\]
\end{proposition}
\begin{proof}
Let $x^*\in T^{\rho}(x)$. If $y\in V_T(x)$ then there exists $y^*\in T(y)$ such that $\inner{x-y}{y^*}>0$. Therefore $(x,x^*)\sim_p (y,y^*)$ so $\inner{y-x}{x^*}<0$. Thus $x^*\in N^{\circ}_{V_T(x)}(x)$, since $y\in V_T(x)$ was arbitrary . On the other hand, if $y\in W_T(x)$, then there exists $y^*\in T(y)$ such that $\inner{x-y}{y^*}=0$. Again, we have $(x,x^*)\sim_p (y,y^*)$, so $\inner{y-x}{x^*}\leq 0$ and $x^*\in N_{W_T(x)}(x)$.

Conversely, take $x^*\in N^{\circ}_{V_T(x)}(x)\cap N_{W_T(x)}(x)$ and $(y,y^*)\in T$. If $\inner{x-y}{y^*}<0$ then clearly $(x,x^*)\sim_p(y,y^*)$.  Now assume that $\inner{x-y}{y^*}=0$, that is, $y\in W_T(x)$. Since $x^*\in N_{W_T(x)}(x)$, $\inner{y-x}{x^*}\leq 0$. This also implies that $(x,x^*)\sim_p(y,y^*)$. Finally, assume that $\inner{x-y}{y^*}>0$, that is $y\in V_T(x)$. In this case, $x^*\in N^{\circ}_{V_T(x)}(x)$ implies $\inner{y-x}{x^*}<0$, which in turn implies $(x,x^*)\sim_p(y,y^*)$. Hence, we conclude $(x,x^*)\in T^{\rho}$.
\end{proof}

\begin{corollary}\label{cor:trhozeros}
Let $T:X\tos X^*$ be a multivalued operator and let $x\in X$. The following are equivalent:
\begin{enumerate}
\item $V_T(x)=\emptyset$,
\item $x\in Z_{T^{\rho}}$,
\item $T^{\rho}(x)=N_{W_T(x)}(x)$,
\item $x\in\dom(T^{\rho})$ and $T^{\rho}(x)$ is weak$^*$-closed.
\end{enumerate}
\end{corollary}
\begin{proof}
If $V_T(x)=\emptyset$ then $N^{\circ}_{V_T(x)}(x)=X^*$ thus implying {\it 3}. Since a normal cone is always non-empty and weak$^*$-closed, {\it 3} implies {\it 4}. Now assume {\it 4} and take $x^*\in T^{\rho}(x)$. If $x^*=0$ then $x\in Z_{T^{\rho}}$, so assume that $x^*\neq 0$. Since $T^{\rho}(x)$ is a cone, $tx^*\in T^{\rho}(x)$, for all $t>0$. Taking the limit when $t\to 0^+$, as $T^{\rho}(x)$ is weak$^*$-closed, we conclude that $0\in T^{\rho}(x)$. Finally, assuming {\it 2}, from Proposition~\ref{pro:trhogen} we obtain $0\in N^{\circ}_{V_T(x)}(x)$, which is absurd, unless $V_T(x)=\emptyset$.
\end{proof}

\section{On pseudomonotone operators}
Recall that an operator $T:X\tos X^*$ is \emph{pseudomonotone} if, for every $(x,x^*),(y,y^*)\in T$, the following implication holds
\[
\inner{y-x}{x^*}\geq0\quad\Longrightarrow\quad \inner{y-x}{y^*}\geq 0,
\]
or, equivalently, every $(x,x^*),(y,y^*)\in T$  are pseudomonotonically related, that is,
\[
\min\{\inner{y-x}{x^*},\inner{x-y}{y^*}\}< 0 \mbox{ or } \inner{y-x}{x^*}=\inner{x-y}{y^*}=0.
\]

As was in the monotone and quasimonotone case, the pseudomonotone polar provides a way of characterizing pseudomonotonicity.
\begin{proposition}\label{pro:4.1} Let $T:X\tos X^*$. The following conditions are equivalent:
\begin{enumerate}
 \item $T$ is pseudomonotone,
 \item $T\subset T^{\rho}$,
 \item $T^{\rho\rho}\subset T^{\rho}$,
 \item $T^{\rho\rho}$ is pseudomonotone,
 \item $\cones(T)$ is pseudomonotone.
\end{enumerate}
\end{proposition}
\begin{corollary}\label{123}
Let $T:X\tos X^*$ be a pseudomonotone operator. Then $(x,x^*)\in T^\rho$ if, and only if, $T\cup\{(x,x^*)\}$ is pseudomonotone.
\end{corollary}

\begin{example}\label{ejem-variational}
 Let $T=(\R\times\{0\})\cup\{(0,1)\}$. Straightforward calculations show that
\[
 T^{\nu}=\{(x,x^*)\in\R^2\::\:x\leq 0\text{ or }x^*\geq 0\}\quad\text{and}\quad
 T^{\rho}=\{(x,0)\in\R^2\::\:x\leq 0\}.
\]
Therefore, $T$ is quasimonotone but not pseudomonotone. Note, in addition, that $T^{\rho}$ is monotone.  
\end{example}


Joining together Propositions~\ref{pro:zeros} and \ref{pro:4.1}, we obtain the following corollary.
\begin{corollary}
Let $T:X\tos X^*$ be a pseudomonotone operator. Then  $\overline{\co}^{w}(Z_T)\subset Z_{T^\rho}$.
\end{corollary}
The latter inclusion is strict
 in general, as shown by the following example.
\begin{example}\label{ejem1}
Let $T:\R\tos\R$ be a multivalued operator defined as $T=(\R_-\times\{-1\})\cup(\R_+\times\{1\})$.
Clearly, $T$ is pseudomonotone and $Z_T=\emptyset$. 
However, 
\[
T^{\rho}=\{(x,x^*)\in\R^2\::\:xx^*>0\text{ or }x=0\}
\]
and $Z_{T^{\rho}}=\{0\}$. 
\end{example} 

\begin{proposition}\label{pseudo-dom} 
Let $T:X\tos X^*$ be a multivalued operator. If 
  $\co(\dom(T))\subset \dom(T^\rho)$ then $T$ is pseudomonotone.
\end{proposition}
\begin{proof}
Take $(x,x^*),(y,y^*)\in T$ and let 
$z=
(x+y)/2 
\in[x,y]\subset\co(\dom(T))$. 
Then $z\in \dom(T^{\rho})$ and there exists $z^*\in T^{\rho}(z)$. In particular, $(z,z^*)$ is pseudomonotonically related to both $(x,x^*)$ and $(y,y^*)$. Therefore, by Proposition~\ref{pro:pseudo-new}, $(x,x^*)$ and $(y,y^*)$ are pseudomonotonically related. The proposition follows.
\end{proof}
%

The following example shows that we cannot drop the condition of taking the convex hull of the domain in the last proposition.

\begin{example}
Let $T=\{(0,1), (1,0)\}$. Then $T^{\rho}=(]-\infty,0]\times\R_-)\cup([1,+\infty[\times \R_{++})$ and $\dom(T)=\{0,1\}\subset\R\setminus]0,1[=\dom(T^{\rho})$. However $T$ is not pseudomonotone.
\end{example}

As a direct consequence of Proposition \ref{pseudo-dom} we have the following corollary.
\begin{corollary}\label{cor:pseudo}
Let $T:X\tos X^*$ be a multivalued operator. If $\dom(T^\rho)=X$ then $T$ is pseudomonotone. 
\end{corollary}
\begin{remark}
Propositions~\ref{pro:pseudo-new} and \ref{pseudo-dom} and Corollary~\ref{cor:pseudo} can be restated for monotonicity instead of pseudomonotonicity, with similar proofs.
\end{remark}

We say that $T$ is \emph{maximal pseudomonotone} if $T$ is pseudomonotone and it is maximal in the sense of inclusion, that is, 
if $T\subset S$ and $S$ is pseudomonotone, then $T=S$.  A direct consequence of Zorn's Lemma shows that every pseudomonotone 
operator must have a maximal pseudomonotone extension, that is, a maximal pseudomonotone operator which contains it.  

\begin{proposition}\label{pro:maxpseu}
	Let $T:X\tos X^*$ be a pseudomonotone operator and let $P(T)$ be the set of maximal pseudomonotone extensions of $T$. Then
	\begin{enumerate}
		\item $T^{\rho}=\displaystyle\bigcup_{M\in P(T)}M$;
		\item $T^{\rho\rho}=\displaystyle\bigcap_{M\in P(T)}M$;
		\item $T$ is maximal pseudomonotone if, and only if, $T=T^{\rho}$.
	\end{enumerate}
\end{proposition}

From Example~\ref{exa:xtimes0}, the operator $T=X\times\{0\}$ satisfies $T=T^{\rho}$, so $T$ is both maximal pseudomonotone and maximal monotone.  The following proposition shows that this operator is the only one with this property.

\begin{proposition}
Let $T:X\tos X^*$ be a monotone operator. Then $T$ is maximal pseudomonotone if, and only if, $T=X\times\{0\}$.
\end{proposition}
\begin{proof}
The ``if'' part follows from Example~\ref{exa:xtimes0} and Proposition~\ref{pro:maxpseu}. We now prove the ``only if'' part.
From the monotonicity and maximal pseudomonotonicity of $T$, we obtain the inclusions $T\subset T^\mu\subset T^\rho=T$, which in turn imply that $T$ is maximal monotone.  

We assert that $Z_T=\dom(T)$. Indeed, let $x\in\dom(T)$ and take $x^*\in T(x)$. Note that the maximal monotonicity and pseudomonotonicity of $T$ together imply that $T^{\rho}(x)=T(x)$ is a weak$^*$-closed conic set. Using Corollary~\ref{cor:trhozeros}, we conclude that $x\in Z_T$. In addition, Proposition~\ref{pro:zeros} implies that $\dom(T)$ is weak-closed and convex.
Assume that there exists $x\notin \dom(T)$.  Using the Hahn-Banach Theorem, there exists $x^*\in X^*$ such that $\inner{y-x}{x^*}<0$, for all $y\in\dom(T)$. This implies that $(x,x^*)\sim_p(y,y^*)$, for all $(y,y^*)\in T$, that is, $(x,x^*)\in T^{\rho}=T$, a contradiction. Therefore $\dom(T)=X$. The proposition follows by noting that $T(x)=N_X(x)=\{0\}$, for all $x\in X$.
\end{proof}

If $T$ is a pseudomonotone operator such that $T^{\rho}$ is pseudomonotone (therefore, maximal pseudomonotone), 
then $T$ will be called \emph{pre-maximal pseudomonotone}.  A pre-maximal pseudomonotone operator is not 
necesarily maximal, but possesses a unique maximal pseudomonotone extension.

\begin{lemma}\label{pre-max}
Let $T,S:X\tos X^*$ be two pseudomonotone operators such that $T\subset S$. If $T$ is pre-maximal pseudomonotone then $S$ is pre-maximal pseudomonotone and $T^{\rho}=S^{\rho}$
\end{lemma}
\begin{proof}
Clearly $T\subset S\subset S^\rho\subset T^\rho$. Since that $T^\rho$ is pseudomonotone its follow that $S^\rho$ is pseudomonotone.
Therefore, $S$ is pre-maximal pseudomonotone.
\end{proof}

For $T:X\tos X^*$ and $\alpha^*\in X^*$, consider the \emph{linear perturbation} $T+\alpha^*:X\tos X^*$ defined as
\[
(T+\alpha^*)(x)=T(x)+\alpha^*,\qquad \forall x\in X.
\] 
The relation between the quasimonotonicity of the linear perturbations $T+\alpha^*$ and the monotonicity of $T$ was originally established by Aussel, Corvellec and Lassonde~\cite{AusCorLass94},
and then extended to include maximality by Aussel and Eberhard~\cite[Proposition~6]{Aussel2013-4}. 
Recently, Bueno and Cotrina presented an analogous version of this result considering pre-maximality~\cite{BueCot16}.
As pre-maximal pseudomonotonicity does not imply pre-maximal quasimonotonicity (see Example~\ref{exa:xtimes0}),
we now present a version of the previous results which relates pre-maximal pseudomonotonicity and monotonicity.
\begin{proposition}
Let $T:X\tos X^*$ be a multivalued operator such that $T+\alpha^*:X\tos X^*$ is pre-maximal pseudomonotone, for all $\alpha^*\in X^*$. Then $T$ is pre-maximal monotone
and $T^\mu$ is pre-maximal pseudomonotone.

\end{proposition}
\begin{proof}
Monotonicity of $T$ is obtained from~\cite[Proposition 2.1]{AusCorLass94}. Following the same proof 
of Proposition 4.12 in \cite{BueCot16} we obtain that $T^{\mu}$ is monotone. 
%
Finally, since $T\subset T^\mu\subset T^\rho$, Lemma \ref{pre-max} and pre-maximality of $T$ imply the 
pre-maximal pseudomonotonicity of $T^\mu$. 
\end{proof}

\section{On $D$-maximal pseudomonotonicity}

We now address a weaker notion of maximality, defined by Hadjisavvas~\cite{Had03}, which he called \emph{$D$-maximal pseudomonotonicity}.  Explicitly, $T:X\tos X^*$ is $D$-maximal pseudomonotone, if $T$ is pseudomonotone and there exists an \emph{equivalent} pseudomonotone operator $S$, that is, $\dom(T)=\dom(S)$, $Z_T=Z_S$ and $\cone(T(x))=\cone(S(x))$, for all $x\in \dom(T)\setminus Z_T$, which has no proper pseudomonotone extension with the same domain.

As expected, a maximal pseudomonotone operator is also $D$-maximal pseudomonotone.

Also in~\cite{Had03}, the author considered the maximal equivalent operator (in the sense of inclusion) of a pseudomonotone operator. Let $T:X\tos X^*$ be a pseudomonotone operator. For $x\in Z_T$, consider
\[
L(T,x)=\{y\in X\::\:\exists y^*\in T(y),\,\inner{x-y}{y^*}\geq 0\},
\]
and the operator $\widehat{T}:X\tos X^*$, 
\[
 \widehat{T}(x)=
 \begin{cases}
 N_{L(T,x)}(x),&\text{if }x\in Z_T,\\
\cone_{\circ}(T(x)),&\text{if }x\in \dom(T)\setminus Z_T,\\
\emptyset,&\text{if }x\notin\dom(T).
\end{cases}                             
\]
\begin{proposition}[{\cite[Proposition 2.1]{Had03}}]\label{pro:had}
Let $T:X\tos X^*$ be pseudomonotone. Then
\begin{enumerate}
\item $\widehat{T}$ is pseudomonotone;
\item $T$ is equivalent to $\widehat{T}$;
\item $\widehat{T}$ is the greatest, in the sense of inclusion, between all pseudomonotone operators equivalent to $T$;
\item if $T$ and $S$ are equivalent, then $\widehat{T}=\widehat{S}$.
\end{enumerate}	
\end{proposition}

Note that $\dom(T)=\dom(\widehat{T})$. In addition, since $Z_T\subset Z_{T^{\rho}}$, 
when $x\in Z_T$, $L(T,x)=W_T(x)$ and $\widehat{T}(x)=T^{\rho}(x)$. This, along with Proposition~\ref{pro:had}, Proposition \ref{pro:polarp}, item {\it 4}, and Proposition \ref{pro:4.1}, item {\it 2}, implies the following inclusions
\begin{equation}\label{eq:1}
T\subset\widehat{T}\subset(\widehat{T})^\rho\subset T^{\rho}.
\end{equation}

The fact that the domain of $\widehat{T}$ is the same as $T$ motivates us to study the pseudomonotone polar of $T$ restricted to $T$'s domain, which we will denote as $T^{\rho}_D=T^{\rho}|_{\dom(T)}$. With this notation, we can rewrite~\eqref{eq:1} and obtain
\begin{equation}\label{eq:2}
T\subset\widehat{T}\subset(\widehat{T})^\rho_D\subset T^{\rho}_D.
\end{equation}
These four operators have the same domain, $\dom(T)$.

We now give a characterization of $D$-maximality in terms of the polar of the operator $\widehat{T}$.
\begin{theorem}\label{teo:dmax}
Let $T:X\tos X^*$ be a pseudomonotone operator. Then $T$ is $D$-maximal pseudomonotone if, and only if, $\widehat{T}=(\widehat{T})^{\rho}_D$.
\end{theorem}

\begin{proof}
Assume that $T$ is $D$-maximal pseudomonotone.
From Proposition~\ref{pro:had}, item~{\it 1}, we have $\widehat{T}\subset(\widehat{T})^{\rho}_D$. 
On the other hand, given $(x,x^*)\in(\widehat{T})^{\rho}_D$, by Corollary~\ref{123}, $\widehat{T}\cup\{(x,x^*)\}$ is pseudomonotone. Since $T$ is $D$-maximal pseudomonotone, $(x,x^*)\in\widehat{T}$. This proves the ``only if'' part. 

For the ``if'' part, note that, by Proposition~\ref{pro:had}, item {\it 3}, it is enough to prove that $\widehat T$ has no proper pseudomonotone extension with the same domain. Otherwise, there would exist $(x,x^*)\in X\times X^*$, with $x\in\dom(T)$, such that $S'=\widehat T\cup\{(x,x^*)\}$ is pseudomonotone. This implies that $(x,x^*)\in(\widehat{T})^{\rho}_D$. Therefore, $(x,x^*)\in\widehat T$ and $T$ is $D$-maximal pseudomonotone. 
\end{proof}

\begin{remark}
	It is not difficult to show that 
	if $\dom(T)=X$ then $T$ is $D$-maximal pseudomonotone if, and only if, $\widehat{T}$ is maximal pseudomonotone. However,
	in the case general is false, consider us the multivalued operator of Example \ref{ejem1}, the operator $\widehat{T}$ is defined by
	\[
	\widehat{T}(x)=\left\lbrace\begin{array}{cc}
	]-\infty,0[&,x<0\\
	]0,+\infty[&,x>0
	\end{array}
	\right.
	\]
	Clearly $T$ is $D$-maximal pseudomonotone but $\widehat{T}$ isnt maximal pseudomonotone.
\end{remark}
Define $T^{\rho\rho}_D=T^{\rho\rho}|_{\dom(T)}$. As a direct consequence of Theorem \ref{teo:dmax} we have the following corollary.
\begin{corollary}
Let $T:X\tos X^*$ be a multivalued operator. If $T$ is $D$-maximal pseudomonotone then $T^{\rho\rho}_D\subset \widehat{T}=(\widehat{T})_D^{\rho\rho}$.  
\end{corollary}
From Theorem \ref{teo:dmax} and inclusion \eqref{eq:2} we have the following corollary.
\begin{corollary}
Let $T:X\tos X^*$ be a $D$-maximal pseudomonotone operator. If $\dom(T)=\dom(T^\rho)$ then 
$\widehat{T}$ is maximal pseudomonotone.
\end{corollary}

\begin{proposition}\label{pro:domconv}
	Let $T:X\tos X^*$ be a pseudomonotone operator. If $\dom(T)$ is convex then $(\widehat{T})^\rho_D=T^{\rho}_D$.
\end{proposition}	
\begin{proof}
From~\eqref{eq:2}, we already have the inclusion $(\widehat{T})^\rho_D\subset T^{\rho}_D$. 
Let $(x,x^*)\in T^{\rho}_D$, then $x\in \dom(T)$ and $(x,x^*)\in T^{\rho}$. 
We now prove that $(x,x^*)\in(\widehat{T})^{\rho}$. 
Indeed, take $(y,y^*)\in\widehat{T}$, then $(y,y^*)\in T^{\rho}$ and $y\in \dom(T)$. 
Since $\dom(T)$ is convex, $z=\dfrac{x+y}{2}\in\dom(T)$ so there exists $z^*\in T(z)$. 
Thus, $(z,z^*)$ is pseudomonotonically related to $(x,x^*)$ and $(y,y^*)$ and, 
by Proposition~\ref{pro:pseudo-new}, $(x,x^*)$ and $(y,y^*)$ are pseudomonotonically related. 
Therefore $(x,x^*)\in (\widehat{T})^{\rho}$, so $(\widehat{T})^\rho_D=T^\rho_D$, and the proposition follows.
\end{proof}
From Theorem~\ref{teo:dmax} and  Proposition~\ref{pro:domconv}, we recover Lemma 3.2 in~\cite{Had03}.
\begin{corollary}\label{H-lemma}
Let $T:X\tos X^*$ be a pseudomonotone operator. If $\widehat{T}=T^{\rho}_D$ then $T$ is $D$-maximal pseudomonotone. In addition, if $\dom(T)$ is convex then the converse is true.
\end{corollary}
\begin{proof}
From the inclusions in~\eqref{eq:2}, if $\widehat{T}=T^{\rho}_D$ then $\widehat{T}=(\widehat{T})^{\rho}_D$ and, by Theorem~\ref{teo:dmax}, $T$ is $D$-maximal pseudomonotone.  When $\dom(T)$ is convex, by Proposition~\ref{pro:domconv}, $T^{\rho}_D=(\widehat{T})^\rho_D$, which coincides with $\widehat{T}$ again from Theorem~\ref{teo:dmax}. 
\end{proof}

It is a direct consequence of Zorn's Lemma that every pseudomonotone operator has a $D$-maximal pseudomonotone extension with the same domain.  Moreover, it is straightforward to prove that $T^{\rho}_D$ is the union of all $D$-maximal pseudomonotone operators containing $T$.

\begin{example}[{\cite{Had03}}]
Consider $T:\R\tos\R$, $T=\{(0,-1),(1,0)\}$, which is a pseudomonotone operator. 
In this case $\widehat{T}=(\{0\}\times \R_{++})\cup(\{1\}\times\R)$. 
Note that $\widehat{T}$ does not have a pseudomonotone extension with the same 
domain, so $T$ is $D$-maximal pseudomonotone.  However, $T^{\rho}_D=(\{0\}\times \R_{+})\cup(\{1\}\times\R)$ 
is not pseudomonotone.  Now consider $S=(\{0\}\times\R_{-})\cup(\{1\}\times\R_+)$, and observe 
that $T\subset S$ and $S$ also does not have a pseudomonotone extension with the same domain. 
Thus $S$ is $D$-maximal pseudomonotone aswell. Finally, observe that $T^{\rho}_D=\widehat{T}\cup S$.
\end{example}

\begin{remark}
If $T$ is pseudomonotone and $\widehat{T}=T^\rho$ then $T$ is pre-maximal pseudomonotone and $D$-maximal pseudomonotone. 
\end{remark}
The condition of pre-maximal pseudomonotonicity can be used to replace convexity of the domain in Corollary~\ref{H-lemma}.
\begin{corollary}
Let $T:X\tos X^*$ be a pre-maximal pseudomonotone operator. Then $T$ is $D$-maximal pseudomonotone if, and only if, $\widehat{T}=T^\rho_D$.
\end{corollary}

\begin{proof}
First note that if $T$ is pre-maximal pseudomonotone then $T^{\rho}_D$ is pseudomonotone.  This, together with~\eqref{eq:2}, implies that $T^{\rho}_D$ is a pseudomonotone extension of $\widehat{T}$.  Therefore, if $T$ is $D$-maximal pseudomonotone, $\widehat{T}=T^{\rho}_D$.  The converse follows again from~\eqref{eq:2} and Theorem~\ref{teo:dmax}.
\end{proof}

From Lemma \ref{pre-max} and inclusion (\ref{eq:1}) we have the following corollary.
\begin{corollary}
Let $T:X\tos X^*$ be a pre-maximal pseudomonotone operator. Then $\widehat{T}$ is pre-maximal pseudomonotone and $(\widehat{T})^\rho=T^\rho$. 
\end{corollary}

\section{On variational inequalities}
We now deal with variational inequality problems, in the sense of Stampacchia and Minty.  Given a multivalued operator $T:X\tos X^*$ and a set $K\subset X$, the \emph{Stampacchia Variational Inequality Problem} associated to $T$ and $K$ is 
\begin{equation}
\text{to find }x\in K,\,\text{such that }\exists\,x^*\in T(x),\,\inner{y-x}{x^*}\geq 0,\,\forall\,y\in K.\tag{SVIP}
\end{equation}
In a dual way, 
the \emph{Minty Variational Inequality Problem} associated to $T$ and $K$ is
\begin{equation}
\text{to find }x\in K,\,\text{such that }\inner{x-y}{y^*}\leq 0,\,\forall\,(y,y^*)\in T,\, y\in K.\tag{MVIP}
\end{equation}
The solution sets of the (SVIP) and (MVIP) will be denoted as $S(T,K)$ and $M(T,K)$, respectively. 

It is not difficult to prove that $M(T,K)$ is convex and (weakly-)closed, provided that $K$ is convex and (weakly-)closed. Also, when $K=X$, 
\[
S(T,X)=Z_T\quad\text{and}\quad M(T,X)=Z_{T^{\rho}}.
\]

We now study the connections between the VIP solution sets and the pseudo and quasi-monotone polars. First, we need the following lemma.
\begin{lemma}\label{lemma1}
Let $T_1,T_2:X\tos X^*$ be multivalued operators and $K\subset X$. If $T_1\subset T_2$ then $S(T_1,K)\subset S(T_2,K)$ and $M(T_2,K)\subset M(T_1,K)$.
\end{lemma}

Noting that $T^\mu\subset T^\rho\subset T^\nu$, by the previous lemma, we conclude 
\begin{gather*}
 S(T^\mu,K)\subset S(T^\rho,K)\subset S(T^\nu,K)=K,\\
 M(T^\nu,K)\subset M(T^\rho,K)\subset M(T^\mu,K).
\end{gather*}
for any $K\subset X$.

The following result says that, in general, there are relationship between solution sets of SVIP and MVIP associated to a multivalued operator and its pseudomonotone polar.

\begin{proposition}\label{S-M}
Let $T:X\tos X^*$ be a multivalued operator. Then, for every $K\subset X$,
\[ 
S(T^\rho,K)\subset M(T,K)\quad\text{and}\quad S(T,K)\subset M(T^\rho,K). 
\]
\end{proposition}
\begin{proof}
Let $K$ be a subset of $X$, take $x\in S(T^\rho,K)$ and assume that $x\notin M(T,K)$. Then there exist $(y,y^*)\in T$, with $y\in K$, and $x^*\in T^{\rho}(x)$, such that $\inner{x-y}{y^*}>0$ and $\inner{y-x}{x^*}\geq 0$. This contradicts the fact that $(x,x^*)\sim_p(y,y^*)$, and we conclude $S(T^\rho,K)\subset M(T,K)$.  The remaining inclusion can be proved analogously.
\end{proof}
The inclusions in the previous result are strict in general. For instance, the operator defined in Example~\ref{ejem1} verifies 
\[
S(T,\R)=\emptyset~\mbox{ and }~M(T^\rho,\R)=\{0\}.
\]
In addition, the operator defined in the Example \ref{ejem-variational} verifies, for $K=\{1,2\}$,
\[
S(T^\rho,K)=\emptyset~\mbox{ and }~M(T,K)=K.
\]

\newcommand{\subsetv}{\rotatebox[origin=c]{270}{$\subset$}}

Note that Proposition~\ref{S-M} does not ask for any condition on $T$ nor $K$. When $T$ is pseudomonotone (or even monotone), we can obtain a more complete array of inclusions between the VIP solution sets.
\begin{corollary}\label{coro:incl}
Let $T:X\tos X^*$ be a multivalued operator and $K\subset X$. 
\begin{enumerate}
\item If $T$ is pseudomonotone,
\begin{equation}\label{eq:incpseudo}
	\begin{array}{ccc}
		S(T,K)&\subset & S(T^{\rho},K)\\
		\subsetv&&\subsetv\\
		M(T^{\rho},K)&\subset &M(T,K)\\
	\end{array}
\end{equation}
\item If $T$ is monotone,
\begin{equation}\label{eq:incmono}
	\begin{array}{ccccc}
		S(T,K)&\subset& S(T^{\mu},K)&\subset& S(T^{\rho},K)\\
		\subsetv&& &&\subsetv\\
		M(T^{\rho},K)&\subset &M(T^{\mu},K)&\subset &M(T,K)\\
	\end{array}
\end{equation}
\end{enumerate}
\end{corollary}

We now recall the following well known result due to John~\cite{John2001}, which characterizes the pseudomonotonicity of a multivalued operator, with certain inclusions between the VIP solution sets.  We must remark that, although John originally dealt with finite dimensional spaces, his proof can be easily adapted to general Banach spaces.

\begin{theorem}[\cite{John2001},Theorem 2]\label{john-pseudomonotone}
	Let $T:X\tos X^*$ be a multivalued operator. Then $T$ is pseudomonotone if, and only if, $S(T,K)\subset M(T,K)$, for all $K\subset \dom(T)$.
\end{theorem}

In a similar way to Theorem \ref{john-pseudomonotone}, we establish a characterization of pseudomonotonicity by comparing the solution sets of the VIPs associated to the operator and its pseudomonotone polar.
\begin{proposition}\label{pseudo-SVIP}
Let $T:X\tos X^*$ be a multivalued operator. Then
the following conditions are equivalent.
\begin{enumerate}
 \item $T$ is pseudomonotone. 
\item $S(T,K)\subset S(T^\rho,K)$, for all $K\subset X$.
\item $M(T^\rho,K)\subset M(T, K)$, for all $K\subset X$.
\end{enumerate}
\end{proposition}
\begin{proof}
From Lemma~\ref{lemma1}, we obtain that {\it 1} implies {\it 2} and {\it 3}. Moreover, from Theorem~\ref{john-pseudomonotone} and equation~\eqref{eq:incpseudo} in Corollary~\ref{coro:incl}, we conclude that items {\it 2} and {\it 3} imply {\it 1}. 
The theorem follows.
\end{proof}

From Theorem~\ref{john-pseudomonotone} and Proposition~\ref{pseudo-SVIP}, we can characterize pre-maximal pseudomonotonicity by the inclusions of the solution sets of the (SVIP) into the (MVIP).
\begin{corollary}
Let $T:X\tos X^*$ be a pseudomonotone operator, everywhere defined. Then $T$ is pre-maximal pseudomonotone if, and only if, $S(T^\rho,K)\subset M(T^\rho,K)$, for all $K\subset\dom(T^{\rho})$.
\end{corollary}

The following result gives a sufficient condition for the pre-maximality of a multivalued operator.
\begin{proposition}
Let $T:X\tos X^*$ be a multivalued operator. If $S(T,K)=S(T^\rho,K)$ or $M(T,K)=M(T^\rho,K)$, for all $K\subset X$ then $T$ is pre-maximal pseudomonotone.
\end{proposition}


\begin{thebibliography}{10}
	
	\bibitem{AusCorLass94}
	D.~Aussel, J.-N. Corvellec, and M.~Lassonde.
	\newblock {Subdifferential characterization of quasiconvexity and convexity}.
	\newblock {\em J. Convex Anal.}, 1(2):195--201 (1995), 1994.
	
	\bibitem{Aussel2013-3}
	D.~Aussel and J.~Cotrina.
	\newblock {Stability of Quasimonotone Variational Inequality Under
		Sign-Continuity}.
	\newblock {\em Journal of Optimization Theory and Applications},
	158(3):653--667, 2013.
	
	\bibitem{Aussel2013-4}
	D.~Aussel and A.~Eberhard.
	\newblock {Maximal quasimonotonicity and dense single-directional properties of
		quasimonotone operators}.
	\newblock {\em Mathematical Programming}, 139(1):27--54, 2013.
	
	\bibitem{AuHad2005}
	D.~Aussel and N.~Hadjisavvas.
	\newblock {Adjusted Sublevel Sets, Normal Operator, and Quasi-convex
		Programming}.
	\newblock {\em SIAM Journal on Optimization}, 16(2):358--367, 2005.
	
	\bibitem{BueCot16}
	O.~{Bueno} and J.~{Cotrina}.
	\newblock {On maximality of quasimonotone operators}.
	\newblock {\em ArXiv e-prints}, Sept. 2016.
	
	\bibitem{Crouzeix1997}
	J.-P. Crouzeix.
	\newblock Pseudomonotone variational inequality problems: Existence of
	solutions.
	\newblock {\em Mathematical Programming}, 78(3):305--314, 1997.
	
	\bibitem{MR2577322}
	J.-P. Crouzeix, A.~Eberhard, and D.~Ralph.
	\newblock {A geometrical insight on pseudoconvexity and pseudomonotonicity}.
	\newblock {\em Math. Program.}, 123(1, Ser. B):61--83, 2010.
	
	\bibitem{Crouzeix2011}
	J.-P. Crouzeix, A.~Keraghel, and N.~Rahmani.
	\newblock Integration of pseudomonotone maps and the revealed preference
	problem.
	\newblock {\em Optimization}, 60(7):783--800, 2011.
	
	\bibitem{Eberhard07}
	A.~C. Eberhard and J.-P. Crouzeix.
	\newblock Existence of closed graph, maximal, cyclic pseudo-monotone relations
	and revealed preference theory.
	\newblock {\em Journal of Industrial and Management Optimization},
	3(2):233--255, 2007.
	
	\bibitem{Had03}
	N.~Hadjisavvas.
	\newblock {Continuity and Maximality Properties of Pseudomonotone Operators}.
	\newblock {\em J. Convex Anal.}, 10(2):459--469, 2003.
	
	\bibitem{HadjSchai06}
	N.~Hadjisavvas and S.~Schaible.
	\newblock On a generalization of paramonotone maps and its application to
	solving the stampacchia variational inequality.
	\newblock {\em Optimization}, 55(5-6):593--604, 2006.
	
	\bibitem{Hadjisavvas2012}
	N.~Hadjisavvas, S.~Schaible, and N.-C. Wong.
	\newblock Pseudomonotone operators: A survey of the theory and its
	applications.
	\newblock {\em Journal of Optimization Theory and Applications}, 152(1):1--20,
	2012.
	
	\bibitem{John2001}
	R.~John.
	\newblock {A Note on Minty Variational Inequalities and Generalized
		Monotonicity}.
	\newblock In N.~Hadjisavvas, J.~E. Mart{\'i}nez-Legaz, and J.-P. Penot,
	editors, {\em {Generalized Convexity and Generalized Monotonicity:
			Proceedings of the 6th International Symposium on Generalized
			Convexity/Monotonicity, Samos, September 1999}}, pages 240--246, Berlin,
	Heidelberg, 2001. Springer Berlin Heidelberg.
	
	\bibitem{Karamardian1976}
	S.~Karamardian.
	\newblock Complementarity problems over cones with monotone and pseudomonotone
	maps.
	\newblock {\em Journal of Optimization Theory and Applications},
	18(4):445--454, 1976.
	
	\bibitem{BSML}
	J.~E. Mart{\'i}nez-Legaz and B.~F. Svaiter.
	\newblock {Monotone operators representable by l.s.c.\ convex functions}.
	\newblock {\em Set-Valued Anal.}, 13(1):21--46, 2005.
	
\end{thebibliography}
\end{document}